\documentclass[11pt]{amsart}
\usepackage[utf8]{inputenc}
\usepackage{amsfonts,amsmath,amssymb,amsxtra,stmaryrd,mathdots}

\usepackage{a4wide}

\usepackage{colonequals}
\usepackage[utf8]{inputenc}
\usepackage[british]{babel}
\usepackage{hyperref}
\hypersetup{colorlinks=true,urlcolor=blue,citecolor=blue,linkcolor=blue}

\usepackage[all,cmtip]{xy}
\usepackage{float}

\usepackage[dvipsnames]{xcolor}
\usepackage[normalem]{ulem}
\usepackage{cleveref}
\Crefname{equation}{}{}

\usepackage{tikz}
\usepackage{tikz-cd}

\usepackage{mathtools}
\usepackage{amsmath}
\usepackage{thmtools}
\usepackage{thm-restate}

\usepackage{tikz-cd}
\usetikzlibrary{patterns,shapes,cd,arrows}
\usepackage[auto]{contour}
\contourlength{2pt}

\usepackage{makecell}

\usepackage{thm-restate}

\usepackage{listings}
\usepackage{xcolor}

\definecolor{codegreen}{rgb}{0,0.6,0}
\definecolor{codegray}{rgb}{0.5,0.5,0.5}
\definecolor{codepurple}{rgb}{0.58,0,0.82}
\definecolor{backcolour}{rgb}{0.95,0.95,0.92}

\lstdefinestyle{mystyle}{
    backgroundcolor=\color{backcolour},   
    commentstyle=\color{codegreen},
    keywordstyle=\color{magenta},
    numberstyle=\tiny\color{codegray},
    stringstyle=\color{codepurple},
    basicstyle=\ttfamily\footnotesize,
    breakatwhitespace=false,         
    breaklines=true,                 
    captionpos=b,                    
    keepspaces=true,                 
    numbers=left,                    
    numbersep=5pt,                  
    showspaces=false,                
    showstringspaces=false,
    showtabs=false,                  
    tabsize=2
}

\usepackage{enumitem}%

\newcommand{\Q}{\mathbb{Q}}

\newcommand{\C}{\mathbb{C}}
\newcommand{\F}{\mathbb{F}}
\newcommand{\PP}{\mathbb{P}}

\newcommand{\Z}{\mathbb{Z}}

\newcommand{\Sy}{\mathcal{S}}

\DeclareMathOperator{\End}{\operatorname{End}}

\DeclareMathOperator{\Hom}{Hom}

\DeclareMathOperator{\Aut}{Aut}

\DeclareMathOperator{\Jac}{Jac}

\DeclareMathOperator{\Spec}{Spec}
\DeclareMathOperator{\Sym}{Sym}

\DeclareMathOperator{\Ind}{Ind}

\DeclareMathOperator{\pr}{pr}

\newcommand{\id}{id}
\DeclareMathOperator{\FinEtCov}{{FinEtCov}}

\usepackage{mathrsfs}

\numberwithin{equation}{section}

\newtheorem{conjecture}[equation]{Conjecture}
\newtheorem{corollary}[equation]{Corollary}

\newtheorem{lemma}[equation]{Lemma}

\newtheorem{proposition}[equation]{Proposition}

\newtheorem{theoremX}{Theorem}

\theoremstyle{definition}
\newtheorem{definition}[equation]{Definition}

\theoremstyle{remark}
\newtheorem{remark}[equation]{Remark}

\newtheorem{example}[equation]{Example}

\usepackage{wrapfig} \title[How to split two-dimensional Jacobians: a geometric construction]{
How to split two-dimensional Jacobians:\\ a geometric construction
}

\author{Andrea Gallese}
\email{andrea.gallese@sns.it}
\address{Affiliation: Scuola Normale Superiore, Pisa, Italy; Address: Piazza dei Cavalieri 7, Pisa, 56125, Italy; ORCID-id: 0009-0007-2105-3397}

\begin{document}

\begin{abstract}
    Let $\pi\colon Y \to X$ be a branched cover of algebraic curves. Assume that there exists a curve $W$ such that
    \[
    \Jac Y \sim \Jac X \times \Jac W.
    \]
    We conjecture that every such isogeny decomposition is induced by an algebraic correspondence of curves that fits in a Galois diagram, and we prove this conjecture when $g(Y)=2$ and $g(X)=1$. Our proof yields a geometric construction of the complementary curve $W$, an explicit correspondence inducing the isogeny, and a general criterion for deciding when an algebraic correspondence of curves fits in a Galois diagram (admits a push-out).
\end{abstract}

\maketitle
\let\thefootnote\relax
\let\thefootnote\relax


\section{Introduction}
Let $Y$ be an algebraic curve of genus $g \geq 2$ defined over an algebraically closed field $k$ of characteristic 0. The Jacobian variety $\Jac Y$ is an abelian variety of dimension $g$. We say that $\Jac Y$ is a \emph{split Jacobian} if it is isogenous to a product of elliptic curves.
Generically, the Jacobian of an algebraic curve does not split as the product of elliptic curves but is rather a simple abelian variety. Moreover, even when it decomposes up to isogeny as the product of simple abelian varieties of lower dimension, these are not necessarily one-dimensional.
Conversely, a product of $g$ elliptic curves is not necessarily isogenous to the Jacobian of a connected curve of genus $g$ \cite{Hayashida1, Hayashida2}.
Therefore, a split Jacobian carries interesting information.

In the present paper, we focus on the case of two-dimensional split Jacobians. 
The two-dimensional case is the easiest non-trivial case to study and, as such, it has been studied thoroughly \cite{Kuhn, MR2039100, MR1962943, magaard2012genus, MR3173190, MR3893776, MR4149056, konstantinou2023}. Nonetheless, it already presents interesting arithmetic applications.
For example, two-dimensional split Jacobians appear in several papers on the height conjecture for elliptic curves \cite{MR1085258, MR1363491, Frey} and in recent articles on the parity conjecture \cite{MR2551757, konstantinou2023}.
They serve as examples in an extensive body of work.

Previous work on the topic has focused on establishing a classification of split Jacobians (for example, by computing explicit models for the moduli space of Jacobians that split via maps of a fixed degree) or criteria to determine whether a given elliptic curve appears as a factor of a split Jacobian.
We pursue a different objective, in the spirit of \cite{lombardo2020decomposing}.

\subsection{The conjecture}
We reformulate the strategy of \cite{lombardo2020decomposing} in a different language, highlighting a novel approach to the problem. 
We always assume that algebraic curves are smooth, projective, and connected, and defined over a field $k \subseteq \C$. Maps between curves are non-constant dominant morphisms, hence branched covers.

\begin{definition}
    We say that a correspondence $\alpha\colon X \leftarrow Z \rightarrow W$ between algebraic curves over $k$ \emph{fits in a Galois diagram} if there exists a morphism $Z' \to Z$ such that both compositions $Z' \to X$ and $Z' \to W$ are Galois covers.
\end{definition}

If \( \alpha \) fits in a Galois diagram and \( g(Z) \geq 2 \), then \( \alpha \) admits a \emph{core}: it dominates the trivial correspondence \( \PP^1 = \PP^1 = \PP^1 \) (see \Cref{section: fitting in a Galois diagram}).

\begin{conjecture}[Every Prym correspondence fits in a Galois diagram]
{\label{conjecture}}
    Let $\pi\colon Y \to X$ be a branched cover of degree $n \geq 2$. Assume that the generalized Prym variety of the cover is isogenous to a Jacobian $\Jac W$. There is an isogeny $\Jac Y \sim \Jac X \times \Jac W$ that is induced by a correspondence of algebraic curves
    \begin{equation}
    {\label{eq: first instance of corr diagram}}
    \begin{tikzcd}
        X & Y \lar["\pi"'] & Z \rar["\alpha_W"]\lar["\alpha_Y"'] & W 
    \end{tikzcd}  
    \end{equation}
    that fits in a Galois diagram.
\end{conjecture}

\Cref{conjecture} suggests searching for a map $x \colon X \to \PP^1$ with the following property. Denote by $\alpha\colon\bar Y \to \PP^1$ the normal closure of $x\pi\colon Y\to X \to \PP^1$. There is a subcover $\alpha_W\colon \bar Y \to W$ of $\alpha$ such that the homomorphism
\[
    \pi_\ast \times (\alpha_W)_\ast (\alpha_Y)^\ast \colon \Jac Y \to \Jac X \times \Jac W
\]
induced by the algebraic correspondence is an isogeny.
This is the approach of \cite{lombardo2020decomposing} where an algorithm is presented that, starting with the cover $\pi$, computationally checks if $W$ can be realized as a quotient of $\bar{Y}$ and whether the correspondence gives the desired isogeny.
Numerous comprehensive tables report the output of the cited algorithm, corroborating the conjecture \cite[Tables 3-9]{lombardo2020decomposing}.

Although any isogeny between Jacobians is induced by a correspondence between the underlying curves \cite[Theorem 11.5.1]{MR2062673}, there is a priori no reason why the correspondence should fit in a Galois diagram; see \Cref{example: correspondence degree 3} and \Cref{remark: the monodromy construction does not give the complementary cover directly}.

\begin{remark}
    We record an interesting question tailored to our discussion.
    For two curves $X$ and $Y$, consider the subspace $\mathfrak{g}$ of $\Hom(\Jac Y,\,  \Jac X) \otimes_\Z \Q$ generated by correspondences that fit in a Galois diagram. Is $\mathfrak{g}$ a proper subspace? 
\end{remark}

\subsection{A concrete case}
\label{section: the problem}
Let $X$ and $Y$ be algebraic curves of genus $g(X) =1$ and $g(Y)=2$, respectively. 
Let $\pi\colon Y \to X$ be a dominant morphism of degree $n \geq 2$.
The covering map $\pi$ induces a non-constant homomorphism between the Jacobian varieties
\[  \pi_\ast\colon \Jac Y \to \Jac X. \]
Since $\Jac X$ is one-dimensional while $\Jac Y$ has dimension two, the connected component of  $\ker\pi_\ast$ is a one-dimensional abelian variety, known in the literature as the generalized Prym variety of $\pi$. Every one-dimensional abelian variety is the Jacobian of a genus-one curve $W$. The Prym variety $\Jac W$ is uniquely determined up to isogeny by the property of fitting in a decomposition
\begin{equation}
    {\label{eq: the hiding elliptic curve}}
    \Jac Y \sim \Jac X \times \Jac W.
\end{equation}
We describe an algorithm that, given $\pi\colon Y \to X$, outputs a representative $W$ for the isogeny class of $\Jac W$; we call $W$ a \emph{complementary curve} for the cover $\pi$. 
We assume without loss of generality that the cover $\pi$ is \emph{irreducible}, i.e.~that it does not admit any intermediate sub-cover. We further assume that there is a Weierstrass point $y_\infty \in Y(\bar k)$ such that $x_\infty = \pi(y_\infty)$ is a $k$-rational point in $X$; this hypothesis can be dropped, as explained in \Cref{remark: dropping rationality}. We fix $x_\infty$ as the zero of an elliptic curve structure on $X$.

\subsection{Original contributions}
We prove that \Cref{conjecture} holds in the setting of \Cref{section: the problem}---for all covers between curves of genera 2 and 1, independently of the degree and the ramification type---and the core is the quotient $x\colon X \to \PP^1$ introduced in \Cref{definition: associate P1 cover phi}.

\begin{theoremX}
\label{theorem: main theorem}
    Let $X$ and $Y$ be smooth projective curves of respective genus $g(X)=1$ and $g(Y)=2$, defined over a field $k$.
    Let $\pi\colon Y \to X$ be an irreducible branched cover of degree $n \geq 2$.
    There is an isogeny $\Jac Y \sim \Jac X \times \Jac W$ that is induced by a correspondence of algebraic curves
    \begin{equation*}
    \begin{tikzcd}
        X & Y \lar["\pi"'] & Z \rar["\alpha_W"]\lar["\alpha_Y"'] & W 
    \end{tikzcd}  
    \end{equation*}
    that fits in a Galois diagram. The core is the hyperelliptic quotient $x\colon X \to \PP^1$.
\end{theoremX}

The proof is divided into three steps: we present a geometric construction that outputs equations for $W$, we show that there is a correspondence between $W$ and $Y$ that induces the desired isogeny, and we show that the correspondence fits in a Galois diagram.

\subsubsection*{A geometric construction}
\label{subsection: geometric construction}
The base points $y_\infty \in Y(\bar k)$ and $x_\infty \in X(k)$ give embeddings of each curve in its Jacobian, and a birational morphism $j \colon \Sym^2 Y \to \Jac Y$.
Composing $j$ with $\pi_\ast\colon \Jac Y \to \Jac X$ we get a rational morphism
\begin{equation*}
    {\label{eq: intro explicit equations for W}}
    \theta\colon \Sym^2 Y \to \Jac X, \quad
    [y_1, y_2] \mapsto [(\pi y_1)+(\pi y_2)-2\cdot (x_\infty)],
\end{equation*}
whose fibers are birational to $\ker(\pi_\ast)$ by construction.
Let $W$ be the normalization of the closure of $\theta^{-1}(x_\infty)$ in $\Sym^2 Y$.

\begin{theoremX}[\Cref{prop: UZ geometry}]
\label{theorem: Prym as a fiber}
    Let $X$ and $Y$ be smooth projective connected curves of genus $g(X)=1$ and $g(Y)=2$, defined over a field $k$.
    Let $\pi\colon Y \to X$ be an irreducible branched cover of degree $n \geq 2$.
    With notation as above, there is a $k$-isomorphism of algebraic curves
    $W \simeq \ker(\pi_\ast)$.
\end{theoremX}

This construction was suggested by Zannier upon reading an earlier version of the present paper and has since been detailed in \cite{gallese2026complement}.
It is recorded here in \Cref{section: new geometric construction} with a presentation adapted to the present context.
The construction also provides useful insight into the next step of the present work, whose content is unchanged from the previous version but whose motivation is now more transparent.

\subsubsection*{A correspondence}
By construction, $W$ is birationally equivalent to the curve in $\Sym^2 Y$ defined by \(\pi(y_1)+\pi(y_2)=0.\)
One gets explicit equations for $W$ by recalling that $\Sym^2 Y$ is the quotient of $Y^2$ by the involution $\eta \colon (y_1, y_2) \mapsto (y_2, y_1)$. This suggests that we can naturally interpret $W$ as a degree-2 quotient of the fiber product $Y\times_XY$. In \Cref{section: the correspondence}, we identify a curve $Z \subseteq Y \times_XY$ and a correspondence
\begin{equation}
    \label{eq: correspondence in the UZ setting intro}
   \begin{tikzcd}
        X & Y \lar["\pi"'] & Z \lar["\pr_1"'] \rar["\beta"] & W,
   \end{tikzcd}
\end{equation}
where $\beta$ is a degree 2 map induced by the projection $Y^2 \to \Sym^2 Y$, $\pr_1$ is the projection of the fiber product onto its first component, and $Z$ is birationally equivalent to the curve in $Y^2$ cut out by the equation \(\pi(y_1)+\pi(y_2)=0\).

\begin{theoremX}
\label{theorem: the correspondence gives a Prym isogeny}
    Let $X$ and $Y$ be smooth projective connected curves of genus $g(X)=1$ and $g(Y)=2$, defined over a field $k$.
    Let $\pi\colon Y \to X$ be an irreducible branched cover of degree $n \geq 2$. With notation as above, correspondence \Cref{eq: correspondence in the UZ setting intro} induces the degree $n^2$ isogeny \[\pi_\ast \times \beta_\ast\pr_1^\ast \colon \Jac Y \to \Jac X \times \Jac W. \]
\end{theoremX}
\begin{proof}
    \Cref{lemma: the natural correspondence retrieves the uz construction} identifies the map $\beta_\ast\pr_1^\ast$ with the inclusion of $\ker(\pi_\ast)$ in $\Jac Y$. The statement follows via the standard argument in \cite[\S 1]{MR1085258}.
\end{proof}

The construction of $W$ and the correspondence in \eqref{eq: correspondence in the UZ setting intro} is algorithmic and can be carried out explicitly to recover algebraic equations for $W$ and the branched covers in the correspondence. We refer to \cite{gallese2026complement} for computational details. 
Here, in \Cref{example: degree 3}, we present a simple case to exemplify the process.

A key step in the proof of \Cref{theorem: the correspondence gives a Prym isogeny} is to classify all possible monodromy groups of the cover $\pi\colon Y \to X$; this is carried out in \Cref{lemma: under mild hypothesis G contains An}.

A standard argument in Galois theory shows that \Cref{theorem: the correspondence gives a Prym isogeny} implies \Cref{theorem: main theorem}.
Indeed, one checks that $\eta\colon (y_1, y_2) \mapsto (y_2, y_1)$ restricts to an involution of the cover $x\pi\pr_1\colon Z \to Y\to X \to \PP^1$.
Rather than providing the details of this argument, we develop a new perspective on the problem.

\subsubsection*{The Prym correspondence fits in a Galois diagram}
\label{subsection: fits in a Galois diagram}
We present a criterion to decide whether a given correspondence $\alpha\colon X \leftarrow Z \to W$ has a core, as per \Cref{definition: fits in a Galois diagram}. Denote by $\alpha_X$ and $\alpha_W$ the two maps in the correspondence $\alpha$. 
Fix a point $z \in Z(\bar k)$ and define inductively the sets
\begin{equation}
    {\label{eq: definition of chaotic set}}
    \mathcal{F}_0(z) := \{z\}, \quad
    \mathcal{F}_{i+1}(z) := \alpha_X^{-1}\left(\alpha_X(\mathcal{F}_i(z))\right) \cup \alpha_W^{-1}\left(\alpha_W(\mathcal{F}_i(z))\right).
\end{equation}
Denote their union by $\mathcal{F}(z) = \bigcup_i \mathcal{F}_i(z)$.

\begin{theoremX}
{\label{theorem: criterion}}
    Let $\alpha\colon X \leftarrow Z \to W$ be a correspondence between smooth projective algebraic curves over a field $k\subseteq \C$. If $\alpha$ has a core, then there is a positive integer $m$ such that all points $z \in Z$ satisfy the inequality 
    \begin{equation}
    {\label{eq: criterion inequality intro}}
        \#\mathcal{F}(z) \leq m.
    \end{equation}
    If there is a positive integer $m$ such that inequality \Cref{eq: criterion inequality intro} holds for infinitely many $z \in Z$, then $\alpha$ has a core.
\end{theoremX}

\begin{proof}
    Follows from \Cref{prop: theorem D in the complex setting} and \Cref{remark: having a core is a geometric property}.
\end{proof}

In \Cref{section: conclusions}, we deduce \Cref{theorem: main theorem} by applying \Cref{theorem: criterion} to the correspondence \cref{eq: correspondence in the UZ setting intro}.
In \Cref{section: related questions}, we discuss related questions that help contextualize the present work in a broader mathematical context.

\subsection{Acknowledgements}
The author is grateful to Davide Lombardo for proposing this fun problem and for his essential guidance throughout all stages of this work. The author also thanks Alberto Casali for his enthusiastic participation in numerous discussions on the geometry of the problem, and Fabrizio Bianchi for clarifying several dynamical aspects of the push-out criterion.
Further thanks go to Alexandrou Konstantinou, Christophe Ritzenthaler, Raju Krishnamoorthy, and Umberto Zannier for their insightful comments on earlier versions of the manuscript.

\section{Construction of the complementary curve}
\label{section: new geometric construction}
Let $\pi\colon Y \to X$ be an irreducible branched cover of smooth projective curves over a field $k$ of respective genus $g=g(Y)$ and $g(X)$. Fix a base point $y_\infty \in Y(\bar k)$ such that $x_\infty = \pi(y_\infty)$ is a $k$-rational point in $X$. The base points $y_\infty$ and $x_\infty$ give embeddings of each curve in its Jacobian.
We compose the birational morphism 
\begin{equation*}
    {\label{eq: natural birational morphism between symmetric power and Jacobian}}
    j\colon \Sym^g Y \to\Jac Y, \qquad (y_1, \dots, y_g)\mapsto [(y_1) + \dots + (y_g) - g\cdot (y_\infty)]
\end{equation*}
with the push-forward $\pi_\ast$ to obtain the morphism
\begin{equation*}
    \theta \colon \Sym^g Y \to \Jac X, \qquad (y_1, \dots, y_g) \mapsto [(\pi y_1)+\cdots +(\pi y_g) -g \cdot (x_\infty)].
\end{equation*}
The rational morphism $j$ restricts to an open embedding on an open subset $U \subseteq \Sym^g Y$. An explicit description of $U$ can be found in \cite[IV.\S 4]{MR770932}. The image $j(U)$ is the complement of the theta divisor in $\Jac Y$. We denote by $\theta^{-1}(x)$ the fiber of the morphism $\theta\colon U \to \Jac X$ over a closed point $x \in \Jac X$. Let $W_x$ denote the normalization of the closure of $\theta^{-1}(x)$ in $\Sym^g Y$.

\begin{proposition}
\label{prop: UZ geometry}
    There is an open subset $U_\theta \subseteq \Jac X$ such that the following hold for all points $x \in U_\theta(\bar k)$.
    \begin{enumerate}
        \item The projective variety $W_x$ is birationally equivalent to $\ker (\pi_\ast)$ over $k(x)$.

        \item Assume that $g(Y) = g(X)+1$. The curve $W_x$ is isomorphic to $\ker (\pi_\ast)$ over $k(x)$.
    \end{enumerate}
\end{proposition}
\begin{proof}
    Part (1) follows from the birationality of $j$ \cite[Theorem 5.1]{MilneJacobians}. As all fibers of a homomorphism of abelian varieties are isomorphic, one has
    \[\pi_\ast^{-1}(x) \simeq \pi_\ast^{-1}(0) = \ker(\pi_\ast). \] 
    Observe that $\ker(\pi_\ast)$ is connected, since $\pi$ is irreducible \cite[\S 1]{MR1085258}. 
    Because $\pi_\ast$ is dominant, the image $\pi_\ast j(U)$ contains an open set $U_\theta$. For all points $x \in U_\theta(\bar k)$ the intersection $j(U) \cap \pi^{-1}_\ast(x)$ is non-trivial and hence open and dense in $\pi_\ast^{-1}(x)$. One thus gets birational morphisms
    \begin{equation*}
    \begin{tikzcd}[column sep = small]
        W_x \rar &
        \theta^{-1}(x) \rar["j"] &
        \pi_\ast^{-1}(x) \cap j(U) \rar &
        \pi_\ast^{-1}(x) \simeq \ker(\pi_\ast).
    \end{tikzcd}
    \end{equation*}
    
    For part (2), notice that $\ker(\pi_\ast)$ and $W_x$ are both algebraic curves. 
    Algebraic curves are birational if and only if they are isomorphic.
\end{proof}

\begin{remark}
\label{remark: blow-down}
    The assumption $g(Y)-g(X)=1$ and the Riemann-Hurwitz formula imply that either $g(X)=1$ and $g(Y)=2$, the setting of \Cref{section: the problem}, or that $g(X)=2$ and $Y$ is a genus-3 degree-2 étale cover of $X$.
    If $g(Y)-g(X) > 1$, the birational equivalence of \Cref{prop: UZ geometry}(1) cannot always be upgraded to an isomorphism. It follows from the analysis in \cite[IV.\S 4]{MR770932} that $\Sym^g Y \to \Jac Y$ is a blow-down of projective spaces; the exceptional divisors are mapped to points in the Jacobian that correspond, via its moduli-space interpretation, to divisors whose linear space have higher dimension than the general one. Therefore $W_x \to \ker(\pi_\ast)$ is a blow-down as well, and its special fibers can be understood in terms of divisors of $Y$.
\end{remark}

\begin{remark}
    \Cref{prop: UZ geometry} gives explicit equations for a representative $W_x$ in the birational class of $\ker(\pi_\ast)$, namely
    \begin{equation}
    \label{eq: divisorial equation for Prym in highe genus case}
        [(\pi y_1)-(x_\infty)] + \dots + [(\pi y_g) - (x_\infty)] = 0,
    \end{equation}
    provided one has access to explicit equations for the addition on the abelian variety $\Jac X$.
    When $g(X)=1$, one identifies $\Jac X$ with the elliptic curve structure on $X$ with base-point $x_\infty$ and the rational functions defining the addition of $\Jac X$ are easily computed~\cite[\S III.2]{MR2514094}. When $g(X) > 1$ explicit equations for the addition are not always easy to write down; see \cite{MR4622408} for the genus-2 case.
\end{remark}

\Cref{prop: UZ geometry} implies that for a generic closed point $x \in \Jac X$, the variety $W_x$ is birationally equivalent to the complementary abelian variety of $\pi\colon Y \to X$. Notice that, for any $x$, the intersection $j(U) \cap \pi_\ast^{-1}(x)$ is either empty or a dense open subset of the fiber. 

\begin{definition}
\label{definition: W}
    In the setting of \Cref{section: the problem}, we denote by $W$ the curve $W_x$ corresponding to $x=x_\infty$. We will show in \Cref{lemma: quotients} that $W$ is one-dimensional or, equivalently, that $\theta^{-1}(x_\infty)$ is non-empty.
\end{definition}

\section{The Prym correspondence}
\label{section: the correspondence}
From this point on, assume that $g(Y)=2$ and $g(X)=1$.
Let $W$ be as in \Cref{definition: W}.
Recall that $\theta^{-1}(x_\infty)$ is an open subset of the curve in $\Sym^2 Y$ defined by
\begin{equation}
\label{eq: equatons for W}
    \pi(y_1)+\pi(y_2)=0,
\end{equation}
where we use $[y_1,y_2]$ as coordinates on $\Sym^2 Y$ and write square brackets to emphasize that the pair is unordered:
$[y_1,y_2] = [y_2,y_1]$.

\begin{definition}
    {\label{def: iota second version}}
    Denote by $\iota \colon Y \to Y$ the hyperelliptic involution on $Y$. We assume that there is a Weierstrass point $y_\infty \in Y(\bar k)$ such that $x_\infty= \pi(y_\infty)$ is $k$-rational; this base point makes $X$ into an elliptic curve such that
    \begin{equation}
    {\label{diagram: definition of iota on Y and its commutativity}}
    \pi \circ \iota = [-1] \circ \pi,
    \end{equation}
    see \cite[p.42]{Kuhn} for details. We denote by $\iota\colon X \to X$ the involution corresponding to $[-1]$.
    
    {\label{definition: associate P1 cover phi}}
    Taking the quotients of $X,\, Y$ by the respective automorphism $\iota$ gives a commutative diagram of branched covers
\begin{equation}{\label{diagram: definition of phi}}
\begin{tikzcd}
    Y \arrow[d, "\pi"] \arrow[r, "t"] & \mathbb{P}^1 \arrow[d, "\phi"] \\
    X \arrow[r, "x"] & \mathbb{P}^1.
\end{tikzcd} 
\end{equation}
The map $\phi$ is known in the literature as the $\PP^1$-covering associated with $\pi$, see \cite[Section 2]{Shaska}, \cite[Proposition 2.2]{MR1962943}, and  \cite[Proposition 2.4]{Frey}.
Notice that diagram \Cref{diagram: definition of phi} is cartesian (in the category of ramified $\PP^1$-covers): indeed, the fiber product $X\times_{\PP^1}\PP^1$ has total degree at least $2n$ over $\PP^1$ (because $x$ is not a subcover of $\phi$) and is dominated by $Y$ (because of the universal property), through a map of degree 1, i.e.~an isomorphism.
\end{definition}

\begin{remark}[Monodromy methods]
    The combined data of the branch locus of $\phi$ and its ramification type uniquely determine the original cover
    $\pi \colon Y \to X$, as explained in \cite[Proposition~2.3]{Frey} (at least for $n \geq 7$).
    This result connects the moduli spaces of the branched covers $\pi \colon Y \to X$ considered in this work with the
    well-established theory of Hurwitz spaces and provides a foundation for a computational approach.
    Details can be found in \cite{Kuhn, MR1363491, Shaska, magaard2012genus}, where this approach has been
    implemented for degrees $n = 2,\, 3,\, 5,$ and $7$.
    An alternative approach to computing explicit models was introduced in \cite{MR3427148} and has been
    implemented for $n \leq 11$.
    The main drawback of these approaches is that they depend on the ramification type of $\phi$, thereby
    splitting the classification problem into several distinct subproblems.
\end{remark}

The hyperelliptic involution is a deck transformation for $x\pi\colon Y\to \PP^1$. Consider the fiber product $Y\times_XY$ in the category $\FinEtCov(\PP^1, \mathcal{B}_\phi)$ of finite covers of $\PP^1$ unramified outside $\mathcal{B}_\phi\subseteq X(\bar k)$, the set of branch points of $x\pi$:
    \begin{equation}
    \label{eq: fiber product askew}
    \begin{tikzcd}
        \rar["\pr_1"] \dar["\pr_2"] Y \times_XY & Y \dar["\pi"] \\
        Y \rar["\pi\iota"] & X. 
    \end{tikzcd}
    \end{equation}
Since $\FinEtCov(\PP^1, \mathcal{B}_\phi)$ is a Galois category~\cite[Definition 2.1, Theorem 5.10, and Exercise 4.4]{cadoret}, fiber products exist and $Y \times_XY$ is well defined up to unique isomorphism.

\begin{definition}
\label{definition: Z}
    We prove in \Cref{corollary: Z has two connected components} that the fiber product in \Cref{eq: fiber product askew} has two connected components 
    \[ Y \times_XY = \Delta Y \sqcup Z. \]
    We denote by $\Delta Y$ the diagonal embedding (induced by $\id_Y,\, \iota$), which is the normalization of the curve defined by $y_1 = \iota(y_2)$ in $Y^2$. We denote by $Z$ the other component.
\end{definition}


Let $\pi_\eta \colon Y^2 \to \Sym^2 Y$ be the natural projection sending an ordered pair $(y_1,y_2)$ to the
corresponding unordered pair $[y_1,y_2]$. This map is the quotient by the involution
$\eta \colon Y^2 \to Y^2$ that swaps the two coordinates.

\begin{lemma}
\label{lemma: quotients}
The involution $\eta$ restricts to each component of $Y\times_XY$. Moreover, it induces
isomorphisms
\[
Y / \langle \eta \rangle \simeq \PP^1
\qquad \text{and} \qquad
Z / \langle \eta \rangle \simeq W,
\]
where $W$ is the curve of \Cref{definition: W}.
In particular, $W=W_{x_\infty}$ is non-empty.
\end{lemma}

\begin{proof}
    It is immediate to check that $\eta$ extends to an involution of $Y\times_XY$ and restricts to the hyperelliptic involution on the diagonal component, and therefore to an involution on $Z$ as well.
    The quotient of $\Delta Y$ by $\eta$ is birational to the curve in $\Sym^2 Y$ defined by the equation
    $y_1 = \iota(y_2)$. Under the map $j\colon \Sym^2 Y \to \Jac Y$, the point $[y_1, y_2]$ is sent to
    \begin{equation}
    \label{eq:point-on-Jacobian-Y}
        [(y_1)+(y_2)-2\cdot(y_\infty)]
        =
        [(\iota y_2)+(y_2)-2\cdot(y_\infty)]
        \in \Jac Y .
    \end{equation}
    The class in \Cref{eq:point-on-Jacobian-Y} is always trivial, since
    \[
        (\iota y_2)+(y_2)-2\cdot(y_\infty)
        = \begin{cases}
            \operatorname{div}(t-t(y_2)) & \text{ if } y_2 \neq y_\infty, \\
            0 & \text{ if } y_2 = y_\infty.
        \end{cases}
    \]
    As the one-dimensional curve (birational to) $\Delta Y/\langle\eta\rangle$ is mapped by $j$ to a
    single point, it must be the exceptional divisor of the morphism
    $j \colon \Sym^2 Y \to \Jac Y$ corresponding to the blow-up of $\Jac Y$ at the origin.
    Since all exceptional fibers of a blow-up are projective spaces, as per \Cref{remark: blow-down}, there is an isomorphism
    \(
    \Delta Y/\langle \eta \rangle \simeq \PP^1.
    \)
    
    On the other hand, the curve $Z$ is birational to the curve in $Y^2 \setminus \Delta Y$ cut out by
    equation~\Cref{eq: equatons for W}.
    Since the hyperelliptic involution on $Y$ is unique, the class
    $[(y_1)+(y_2)-2\cdot(y_\infty)]$ is nonconstant as $(y_1,y_2)$ varies in $Z$, and hence
    $\theta^{-1}(x_\infty)$ is nonempty.
    Therefore, the projection $\pi_\eta$ induces a rational map
    \(
    Z \dashrightarrow \theta^{-1}(x_\infty).
    \)
    By \Cref{prop: UZ geometry}, composing further yields a rational map
    \[
    Z \dashrightarrow \theta^{-1}(x_\infty) \dashrightarrow W.
    \]
    Since both $Z$ and $W$ are algebraic curves, this rational map extends to a morphism
    $\beta \colon Z \to W$ of degree~$2$.
\end{proof}

From \Cref{lemma: quotients}, we get a correspondence
\begin{equation}
    \label{eq: correspondence in the UZ setting}
   \begin{tikzcd}
        X & Y \lar["\pi"'] & Z \lar["\pr_1"'] \rar["\beta"] & W.
   \end{tikzcd}
\end{equation}
We prove in the following lemma that correspondence \Cref{eq: correspondence in the UZ setting} induces the natural inclusion map $\Jac W \to \Jac Y$ that defines $\Jac W$ as the kernel of $\pi_\ast\colon \Jac Y \to \Jac X$.

\begin{lemma}
\label{lemma: the natural correspondence retrieves the uz construction}
    The following diagram is commutative
    \begin{equation*}
    \begin{tikzcd}
        \theta^{-1}(x_\infty) \rar[hook] \dar[dashed]
        & \Sym^2 Y \rar["\theta"] \dar["j"]
        & \Jac X \dar[equal] \\
        \Jac W \rar["(\pr_1)_\ast\beta^\ast"] & \Jac Y \rar["\pi_\ast"] & \Jac X.
    \end{tikzcd}
    \end{equation*}
\end{lemma}
\begin{proof}
    The right square commutes by definition of $\theta$.
    We can identify a point of $\theta^{-1}(x_\infty)$ with an unordered pair $[y_1, y_2] \in \Sym^2 Y$ such that $\pi(y_1)+\pi(y_2) = 0$ and $y_2 \neq \iota(y_1)$.
    This is sent to a point in $\Jac W = W$ through the vertical map $\theta^{-1}(x_\infty) \dashrightarrow \Jac W$. This point is the image of $(y_1, y_2) \in Z \subseteq Y \times_XY$ through $\beta\colon Z \to W$; we denote its class by $[y_1, y_2]$.
    Directly from our construction, we have that
    \[ \beta^\ast[y_1, y_2] = [(y_1, y_2) + (y_2, y_1) -2\cdot (z_\infty)] \in \Jac Z, \]
    where each $(y, y')$ is an ordered pair in $Z \subseteq Y^2$ and $z_\infty$ is a basepoint on $Z$ such that $\pr_1(z_\infty)=y_\infty$. It follows that
    \[(\pr_1)_\ast\beta^\ast[y_1, y_2] = [(y_1)+(y_2) -2\cdot (y_\infty)] \in \Jac Y, \]
    as claimed.
\end{proof}

\Cref{lemma: the natural correspondence retrieves the uz construction} identifies the map induced by the correspondence \cref{eq: correspondence in the UZ setting} with the inclusion of $\ker(\pi_\ast)$ in $\Jac Y$, proving \Cref{theorem: the correspondence gives a Prym isogeny}. 

\begin{remark}
\label{remark: dropping rationality}
    The assumption that the base points $x_\infty$ and $y_\infty$ be $k$-rational can be dropped.
    Indeed, if $\pi\colon Y \to X$ is defined over $k$, then the set $\mathcal{B}_\phi$ and the involution $\eta$ are defined over $k$.
    It follows that the curve $W$ can be constructed over $k$ by forming the fiber product in \Cref{eq: fiber product askew} and applying the projection of \Cref{lemma: quotients}.
\end{remark}

\begin{example}
{\label{example: degree 3}}
    As already mentioned, \cite[Section 6]{Kuhn} gives an explicit presentation of the moduli space of all covers $\pi\colon Y \to X$ of degree 3 with $g(Y)=2$ and $g(X)=1$. We select the point $(a,b,c)=(0,0,-4)$ in this moduli space (the reader may tweak this parameter in our script \cite{OurScripts} to get more examples).
    It corresponds to the cover $\pi\colon Y \to X$, where
    \begin{equation*}
        \label{eq: generic case example degree 3}
        \begin{cases}
            Y \colon    & s^2 = (t^3-1)(t^3-4),    \\
            X \colon    & y^2 = f(x) =  27x^3+1,
        \end{cases}
    \end{equation*}
    and the associated $\PP^1$-cover is $\phi(t) = {\phi_N(t)}/{\phi_D(t)} = {t^2}/{(t^3-4)}$.
    Since $Y$ is the limit of $x,\, \phi\colon X \to \PP^1 \leftarrow \PP^1$ (we explain in \Cref{definition: associate P1 cover phi} that diagram \Cref{diagram: definition of phi} is cartesian), we can also realize $Y\times_XY$ as the limit of the following diagram
    \begin{equation}
    \label{diagram: zigzag for pullback}
    \begin{tikzcd}[column sep = small]
    \mathbb{P}^1 \arrow[rd, "\phi"'] &&
    X \arrow[ld, "x"'] \arrow[rd, "x"] &&
    \mathbb{P}^1 \arrow[ld, "\phi"] \\
    & \mathbb{P}^1 && \mathbb{P}^1. &
    \end{tikzcd}
    \end{equation}
    An affine model for $Y \times_XY$ as the limit of diagram \Cref{diagram: zigzag for pullback} is given by the ideal in $k[t_1, t_2, y, x]$ generated by
    \[
        \phi_N(t_1) - x \cdot \phi_D(t_1), \quad 
        \phi_N(t_2) - x \cdot \phi_D(t_2), \quad
        y^2 = f(x).
    \]
    With the help of a computer algebra system, it is easy to check that $Y \times_XY$ has two connected components: $\Delta Y$ has genus 2, while $Z$ has genus 4. Their quotients by $\eta$ are, respectively, $\PP^1$ and an elliptic curve $W$ with $j(W)=0$. The result agrees with \cite[Section 6]{Kuhn}.
\end{example}

\section{Fiber product of a cover with itself}
{\label{section: construction of the product}}
The goal of this section is to prove that the fiber product of \Cref{eq: fiber product askew} has two connected components. We base change to the complex field and assume $k = \C$ for the entire section, as this streamlines some geometric arguments; we explain in the proof of \Cref{corollary: Z has two connected components} how to remove this hypothesis. As a byproduct, we compute the possible monodromy groups of the cover $\pi$.

Let $X$ and $Y$ be complex connected smooth projective curves of respective genera $g(X)=1$ and $g(Y)=2$.
Let $\pi\colon Y \to X$ be a branched cover of degree $n$.
Assume that $\pi\colon Y \to X$ is irreducible: if $\pi$ splits as the composition of two morphisms, one of the two has degree 1.
The branched cover $\pi$ is étale outside a finite set $\mathcal{B}\subseteq X$, known as the branch locus of $\pi$ \cite[\href{https://stacks.math.columbia.edu/tag/0C1C}{Lemma 0C1C}]{stacks-project}.
\begin{remark}
{\label{remark: the famous 3 cases}}
    Let $e_\pi(y)$ denote the ramification index of a point $y \in Y$ through the map $\pi$. The Riemann-Hurwitz formula gives
    \[ \textstyle \sum_{y \in Y} (e_\pi(y)-1) = (2g(Y)-2)-n\cdot (2g(X)-2) = 2. \]
    There are only 3 ramification types compatible with this computation.
    \begin{enumerate}[label=\Roman*.]
        \item There are two ramified points $y_1, \, y_2 \in Y$ with ramification index 2 in different fibers: $e_\pi(y_1)=e_\pi(y_2)=2$ and $\pi(y_1)\neq \pi(y_2)$. This is the generic ramification type. 
        
        \item There are two ramified points $y_1, \, y_2 \in Y$ with ramification index 2 in the same fiber: $e_\pi(y_1)=e_\pi(y_2)=2$ and $\pi(y_1) = \pi(y_2)$.
        
        \item There is a single ramified point $y \in Y$ with ramification index 3: $e_\pi(y)=3$.
    \end{enumerate}
\end{remark}

\begin{definition}
\label{definition: fiber product not askew}
    Let $Y\times_XY$ be the following fiber product in the category $\FinEtCov(X, \mathcal{B})$:
    \begin{equation}
    {\label{diagram: definition of Z}}
    \begin{tikzcd}
    Y\times_XY  \rar["\pr_1"]\dar["\pr_2"] & Y \dar["\pi"] \\
    Y \rar["\pi"] & X.
    \end{tikzcd}
    \end{equation}
    This is formally different from the fiber product of \Cref{eq: fiber product askew}, but the two products are isomorphic. 
    Via the equivalence between branched covering spaces and étale algebras of \cite[Theorem 6.16]{cadoret}, we can realize $Y\times_XY$ as the normalization of $X$ in $\Spec\left[ k(Y) \otimes_{k(X)} k(Y) \right]$.
\end{definition}

\subsection{The connected components}
The number of connected components of $Y\times_XY$ is equal to the number of factors of the corresponding étale $k(X)$-algebra $k(Y)\otimes_{k(X)}k(Y)$. The following lemma in commutative algebra determines the number of components.

\begin{lemma}
{\label{lemma: L tensor L over K}}
    Fix a field $K$ and a separable irreducible polynomial $f \in K[t]$, and consider the field extension $L=K[t]/(f)$. Let $f = g_1g_2\cdots g_r$ be the factorization of $f$ into irreducible polynomials over $L$. There is an isomorphism of $K$-algebras
    \[ L \otimes_K L \simeq \prod_{i=1}^r \frac{L[t]}{(g_i(t))}. \]
\end{lemma}

Concretely, when the cover $\pi$ is given explicitly, one can compute a polynomial $f$ such that $k(Y)\simeq k(X)[t]/(f)$, its decomposition into simple factors over $L=k(Y)$, as in the statement of \Cref{lemma: L tensor L over K}, and therefore the number of connected components of $Y\times_XY$. 

\begin{example}[the Galois case]
\label{remark: the Galois case I}
    Suppose that $\pi\colon Y \to X$ is a normal cover or, equivalently, that $k(Y)/\pi^\ast k(X)$ is a normal extension with Galois group $G$.
    This is an extremely restrictive condition: if $\pi$ is normal, all points in a fiber $\pi^{-1}(x)$ have the same ramification degree \cite[Proposition 3.2.10]{Tamas}.
    From the computation in \Cref{remark: the famous 3 cases} it follows that there are only a few possible ramification types for $\pi$. 
    The degree $n$ of $\pi$ (which is equal to the cardinality of $G$) is bounded: at most, a fiber containing a ramified point, contains two doubly ramified ones, and therefore $n \leq 4 $. Furthermore, notice that $n \neq 4$, since $G$ must be simple (and $\pi$ is irreducible). Therefore $n \leq 3$.
    Also, we have an isomorphism
    \begin{equation*}
    {\label{eq: KY tensor KY in the Galois case}}
    \begin{aligned}
        k(Y) \otimes_{k(X)} k(Y) &\simeq \textstyle \prod_{g \in G} k(Y) \\
        (y_1, y_2) &\mapsto (y_1 \cdot g(y_2))_{g},
    \end{aligned}  
    \end{equation*}
    which proves that $Y\times_XY$ is the disjoint union of $n$ copies of $Y$, and $G$ is cyclic of order $2$ or $3$.
\end{example}

The situation is more complex when the cover $\pi\colon Y \to X$ is not normal.
Let $G$ be the Galois group of (the Galois closure of) $\pi\colon Y \to X$. Recall that, via Galois theory, the group $G$ identifies naturally to a subgroup of $\Sy_n$, up to conjugation.
We compute the possible conjugacy classes for $ G \subseteq \Sy_n$ in the following lemma.
When the degree $n$ of $\pi$ is large enough, $G$ is either the alternating group $\mathcal{A}_n$ or the symmetric group $\Sy_n$. There is an exceptional subgroup in degree $n=6$, isomorphic to $\Sy_5$ and transitive (such a subgroup exists because $\Sy_6$ has an outer automorphism \cite[Section 8.2]{dixon}).

\begin{lemma}
{\label{lemma: under mild hypothesis G contains An}}
    Assume the cover $\pi\colon Y \to X$ of \Cref{section: the problem} is irreducible of degree $n \geq 4$. The Galois group $G$ of the Galois closure of $\pi$ is conjugate to one of the following subgroups of $\Sy_n$.
    \begin{itemize}
        \item If $n \neq 6$, then $G$ is conjugate to either the symmetric group $\Sy_n$ or the alternating group $\mathcal{A}_n$.
        \item If $n = 6$, then $G$ is conjugate to either the symmetric group $\Sy_n$, or the alternating group $\mathcal{A}_n$, or the exceptional subgroup
        \[ \Sy_5 \simeq \operatorname{PGL}(2, \F_5) \simeq \langle \, (3546), \, (162)(345) \,  \rangle \subseteq \Sy_6. \]
    \end{itemize}
     
\end{lemma}

We prove that $G\subseteq \Sy_n$ is a \emph{primitive} transitive subgroup (following the definition of \emph{primitive} given in \cite[Section 1.5]{dixon}), with an index $n$ maximal subgroup $H\subseteq G$. Depending on the ramification type of $\pi$, we find elements in $G$ of which we know the cycle type. Group theory allows us to identify the few subgroups of $\Sy_n$ satisfying these conditions.

\begin{proof}  
    Let $k(\bar Y)/\pi^\ast k(X)$ be the Galois closure of $k(Y)/\pi^\ast k(X)$.
    The Galois group $G$ of $\pi\colon Y \to X$ is the Galois group of $k(\bar Y)/\pi^\ast k(X)$; this follows from the equivalence of \cite[Theorem 6.16]{cadoret}.
    Since $k(X)$ is perfect, $\pi^\ast$ is separable, and there is a primitive element $\alpha_1 \in k(Y)$ such that $k(Y)=k(X)(\alpha_1)$. Let $f \in k(X)[t]$ be the minimal (irreducible) polynomial of $\alpha_1$. Then $k(Y)\simeq k(X)[t]/(f)$. Furthermore, the Galois group $G$ acts transitively on the roots $\alpha_1,\,\alpha_2\, \dots, \,\alpha_n$ of $f$ and identifies through this action to a transitive subgroup of $\Sy_n$ 
    \cite[Theorem 11.6]{Garling}.

    Let $H \subseteq G$ be the subgroup corresponding to the extension $k(Y)/\pi^\ast k(X)$ via Galois theory.
    The cover $\pi$ being irreducible (and non-Galois, see \Cref{remark: the Galois case I}) translates into $H$ being maximal (and non-trivial). Notice that $H$ is the stabilizer of $\alpha_1$ in $G$. The maximality of this stabilizer is equivalent to $G$ being primitive, by \cite[Corollary 1.5A and the following paragraph]{dixon}.

    Fix a point $x \in X$ and $x \not\in\mathcal{B}$.
    Via the equivalence of \cite[Chapter III, Proposition 4.9]{mir}, the cover $\pi\colon Y \to X$ is determined by the monodromy representation
    \begin{equation}
    {\label{eq: monodromy representation in proof}}
    \rho\colon \pi_1(X - \mathcal{B}, x) \to \operatorname{Sym}\left(\pi^{-1}(x)\right) \simeq \Sy_n.
    \end{equation}
    Fix a branch point $b \in \mathcal{B}$ and let $\gamma_b$ be a small loop based at $x$ around $b$.
    The image of the representation $\rho$ is a set of generators for the Galois group $G \subseteq \Sy_n$ of $\pi$, by \cite[Theorem 3.10.(i)]{lombardo2020decomposing}.
    The cycle type of $\rho([\gamma_b])$ is the ramification type of the fiber $\pi^{-1}(b)$, as in the statement of \cite[Chapter III, Proposition 4.9]{mir}. There are three possibilities for the ramification structure of $b$, compatible with the Riemann-Hurwitz formula, as explained in \Cref{remark: the famous 3 cases}:
    \begin{enumerate}[label=\Roman*.]
        \item There are two ramified points $y_1, \, y_2 \in Y$ with ramification degree 2 in different fibers: $\pi(y_1)\neq \pi(y_2)$.
        Choosing $b = \pi(y_1)$, we get a transposition $\rho([\gamma_b])$ in $G$. By the standard group-theoretic argument of \cite[Theorem 3.3A.(ii)]{dixon}, a transitive primitive subgroup $G\subseteq\Sy_n$ with a transposition is the full symmetric group $\Sy_n$.
    \end{enumerate}
    In cases II and III the branch locus $\mathcal{B}$ of $\pi$ contains a single point $b$. The topological space underlying $X-b$ is a torus minus a point. Its fundamental group is the free group on two generators. Indeed, if we let $\gamma_1,\, \gamma_2$ be the two paths in $X-b \subseteq X$ whose classes are generators for $\pi_1(X, x)$, they satisfy the commutativity relation
    \begin{equation}
    {\label{eq: relation in pi1 of torus}}
        [\gamma_1][\gamma_2][\gamma_1]^{-1}[\gamma_2]^{-1} = 1.        
    \end{equation}
    When we remove a point, the resulting topological space retracts on $\gamma_1 \cup \gamma_2$, the wedge sum of two circles at the base point $x$. We gain a generator $[\gamma_b]$, and relation \Cref{eq: relation in pi1 of torus} transforms into
    \[ [\gamma_1][\gamma_2][\gamma_1]^{-1}[\gamma_2]^{-1} = [\gamma_b]. \]
    In particular, the group $G \subseteq \Sy_n$, is generated by $\rho([\gamma_1])$ and $\rho([\gamma_2])$ and the permutation type of their commutator is prescribed by the ramification type of $\pi^{-1}(b)$. This imposes a further group-theoretic condition on $G$.
    \begin{enumerate}[label=\Roman*.]
    \setcounter{enumi}{1}
        \item There are two ramified points $y_1, \, y_2 \in Y$ with ramification degree 2 in the same fiber: $\pi(y_1) = \pi(y_2)$.
        Choosing $b = \pi(y_1)$, we get a double transposition $\rho([\gamma_b])$ in $G$.
        By \cite[Example 3.3.1 and Theorem 3.3A.(ii)]{dixon}, when $n \geq 9$, a transitive primitive subgroup $G\subseteq\Sy_n$ with a double transposition contains the alternating group $\mathcal{A}_n$: \cite[Example 3.3.1]{dixon} proves that the minimal degree of $G$ is less than 4, hence it is $2$ or 3 (i.e.~the group contains either a 2-cycle or a 3-cycle), and \cite[Theorem 3.3A]{dixon} applies in either case.
        For $n \leq 8$ there is one exceptional subgroup of $\Sy_n$ satisfying the group-theoretic conditions found so far. With the help of the computer algebra system \textsc{Magma} we list all primitive subgroups of $\Sy_n$, and select the ones having a maximal subgroup of index $n$ and such that the commutator of two generators is a double transposition. The code in \cite{OurScripts} outputs the exceptional subgroup in degree $n = 6$.

        \smallskip
        
        \item There is one ramified point $y \in Y$ with ramification degree 3.
        Choosing $b = \pi(y)$, we get a 3-cycle $\rho([\gamma_b])$ in $G$.
        By \cite[Theorem 3.3A.(i)]{dixon}, a transitive primitive subgroup $G\subseteq\Sy_n$ containing a 3-cycle also contains the alternating group $\mathcal{A}_n$.
    \end{enumerate}
    Apart from the exceptional subgroup in degree $n =6$, the Galois group $G\subseteq \Sy_n$ contains the alternating group $\mathcal{A}_n$ as a subgroup.
    Since $\mathcal{A}_n$ is a maximal subgroup of $\Sy_n$ (its index is 2), the Galois group $G$ is isomorphic to either $\Sy_n$ or $\mathcal{A}_n$.
\end{proof}

\begin{remark}[on the exceptional case]
    In cases II (resp.~III) we imposed all possible restrictions. Let $G$ be a primitive (transitive) subgroup $G \subseteq \Sy_n$, with a maximal subgroup of index $n$ and two generators $g_1$, $g_2$ whose commutator is a double transposition (resp.~a 3-cycle). Consider the homomorphism
    \[ \rho\colon \pi_1(X-b, x) \to G, \qquad [\gamma_i] \mapsto g_i \quad i = 1,\, 2. \]
    This representation corresponds via \cite[Chapter III, Proposition 4.9]{mir} to a branched (connected) cover $\pi_\rho\colon Y \to X$ ramified over the single point $b$. Since
    \[ \rho([\gamma_b])=\rho([\gamma_1][\gamma_2][\gamma_1]^{-1}[\gamma_2]^{-1}) =g_1g_2g_1^{-1}g_2^{-1},  \]
    the ramification type of $\pi_\rho$ is that of case II (resp.~III). Applying the Riemann-Hurwitz formula, it is easy to check that $g(Y)=2$. This implies in particular that there exists a cover $\pi\colon Y \to X$ of degree $n=6$ with Galois group $G \simeq \operatorname{PGL}(2, \F_5)$.
\end{remark}

\begin{remark}[on the Galois case in degree $n=3$]
{\label{remark: degree 3 Galois case final death}}
    Suppose $\pi\colon Y \to X$ is Galois and $n=3$.
    The argument in the proof of \Cref{lemma: under mild hypothesis G contains An} applies: we are in case III (as explained in \Cref{remark: the Galois case I}), the monodromy representation $\rho$ in \Cref{eq: monodromy representation in proof} determines the cover, and the image of $\rho$ generates $G \simeq \Z/3\Z$. Since $G$ is commutative, there are no generators $g_1$ and $g_2$ in $G$ whose commutator is non-trivial.
    It follows that if $n=3$, then $\pi\colon Y \to X$ is not normal and the Galois group of its Galois closure is isomorphic to $\Sy_3$.
\end{remark}

The natural action of $G \subseteq \Sy_n$ on a set with $n$ elements is enough to determine the Galois action on $k(\bar Y)/\pi^\ast k(X)$ and the genus of each connected component of $Y \times_XY$, as explained in \Cref{lemma: L tensor L over K}. In particular, we can use the classification of \Cref{lemma: under mild hypothesis G contains An} to prove that the number of connected components of $Y \times_XY$ is $2$.

\begin{proposition}
    {\label{lemma: compute r in the generic case}}
    Let $G$ be the Galois group of (the Galois closure of) the irreducible branched cover $\pi\colon Y \to X$. Suppose $n \geq 3$.
    The curve of \Cref{definition: fiber product not askew} has two connected components $Z_1$ and $Z_2$. Up to renaming, $Z_1$ is the image of the diagonal embedding $Y\hookrightarrow Y \times_X Y$, while $g(Z_2)> g(Z_1) = 2$. 
\end{proposition}

\begin{proof}
    The irreducible cover $\pi\colon Y \to X$ of degree $n$ corresponds via Galois theory to a maximal subgroup $H \subseteq G$ whose index is $[G:H]=n$. We divide the proof into separate cases, depending on the conjugacy class of $G\subseteq\Sy_n$. \Cref{lemma: under mild hypothesis G contains An} states that $G$ is the symmetric group $\Sy_n$, or the alternating subgroup $\mathcal{A}_n$, or the exceptional group $\operatorname{PGL}(2, \F_5)$. For $n =3$, we proved in \Cref{remark: degree 3 Galois case final death} that $G \simeq \Sy_3$.
    Let $\alpha_1$ be a primitive element for the extension $k(Y)/\pi^\ast k(X)$, let $f$ be the minimal polynomial of $\alpha_1$ over $k(X)$, and recall that the action of $G$ on the roots of $f$ determines the embedding $G \subseteq \Sy_n$ and vice versa. %

    \smallskip

    \emph{Case 1.}
    Suppose that $G \subseteq \Sy_n$ is either $\Sy_n$ or $\mathcal{A}_n$. The subgroup $H\subseteq G$ is the stabilizer of $\alpha_1$: it is therefore 
        $\Sy_{n-1} \subseteq \Sy_n$ or $\mathcal{A}_{n-1} \subseteq \mathcal{A}_n$.
        Through the natural action of $G$ on the roots of $f$, the subgroup $H$ fixes $\alpha_1$ while acting transitively on the remaining $(n-1)$ roots (this is not always the case for $n \leq 3$: the relevant exception is $G \simeq \mathcal{A}_3 \simeq \Z/3\Z$ and $H \simeq \mathcal{A}_2 \simeq 1$. We excluded this case in \Cref{remark: degree 3 Galois case final death}). We get the factorization $f(t) = (t-\alpha_1)g(t)$ over $k(Y)$, where $g(t)$ is an irreducible polynomial of degree $(n-1)$.
        \Cref{lemma: L tensor L over K} translates the decomposition of $f$ into the decomposition of the étale $k(Y)$-algebra
        \[ k(Y)\otimes_{k(X)}k(Y) \simeq k(Y) \times L, \]
        where $L/k(Y)$ is a field extension of degree $(n-1)$.
        Via the correspondence between étale algebras and branched coverings of \cite[Theorem 6.16]{cadoret}, the trivial extension $k(Y)/k(Y)$ corresponds to the trivial cover $\operatorname{id}_Y\colon Y \to Y$, while $L/k(Y)$ corresponds to a connected branched cover $Z_2 \to Y$ of degree $(n-1)$. An application of the Riemann-Hurwitz formula to $\beta\colon Z_2 \to Y$ reveals that
        \[ g(Z_2)-1 = (n-1)(g(Y)-1) + \tfrac{1}{2}\cdot R_\beta > g(Y) -1, \]
        since $(n-1)\geq 2$ (here $R_\beta$ is the degree of the ramification divisor of $\beta$).

        \smallskip
        
        \emph{Case 2.} Let $G\simeq \operatorname{PGL}(2, \F_5)$. With the help of the computer algebra system \textsc{Magma} we compute the orbit decomposition of the natural action of all maximal subgroups $H \subseteq G$ with index $n=6$ on a set with $6$ elements. From the output of the code in \cite{OurScripts} we conclude that $f(t)$ factors as a product $(t-\alpha_1)g(t)$ over $k(Y)$, where $g(t)$ is an irreducible polynomial of degree $5$. The same conclusion of Case 1 follows.
\end{proof}


\begin{corollary}
\label{corollary: Z has two connected components}
    The product $Y\times_XY$ in \Cref{definition: Z} has two connected components.
\end{corollary}
\begin{proof}
    Notice that $\Delta Y$ (the image of $Y$ via $\id_Y, \, \iota$) is always a connected component defined over $k$, which does not coincide with the whole space. After base-change, the number of connected components can only increase. Since there are already at least 2 connected components, we can assume that $k = \C$.
    The statement follows from \Cref{lemma: compute r in the generic case} and the isomorphism $r$ between the fiber products of \Cref{definition: Z} and \ref{definition: fiber product not askew}:
    \begin{equation*}
    \begin{tikzcd}
        Y \dar["\pi"] & Y \times_XY \dar\lar & Y\times_XY \dar\lar["r"] \\
        X & Y \lar["\pi"] & Y \lar["\iota"].
    \end{tikzcd}
    \end{equation*}
\end{proof}

\section{A criterion for fitting in a Galois diagram}
{\label{section: fitting in a Galois diagram}}
We present a criterion for establishing whether a given algebraic correspondence has a core.
This notion was introduced in \cite{Raju} and further studied in \cite{Joel}.

\begin{definition}
\label{definition: fits in a Galois diagram}
    We say that a correspondence between curves $\alpha\colon Y \leftarrow Z \to W$ over a field $k$ \emph{has a core} if the intersection $\alpha_Y^\ast k(Y) \cap \alpha_W^\ast k(W)$ has finite index in $k(Z)$.
\end{definition}

\begin{remark}
\label{remark: having a core is a geometric property}
    Having a core is a geometric property, as it is equivalent to $\alpha_Y^\ast k(Y) \cap \alpha_W^\ast k(W)$ having transcendence degree 1 over $k$. Specifically, if $\alpha_Y^\ast \bar k(Y) \cap \alpha_W^\ast \bar k(W)$ has finite index in $\bar k(Z)$ then $\alpha_Y^\ast k(Y) \cap \alpha_W^\ast k(W)$ has finite index in $k(Z)$.
    
    We thus replace \( k \subseteq \C \) by \( k=\C \) for the remainder of the section, and assume that \( \alpha \) is a correspondence between smooth compact Riemann surfaces, or equivalently smooth complex algebraic curves.
    We will freely pass between the analytic and algebraic viewpoints, using the identification provided by the GAGA principle~\cite[Appendix~B, Theorem~2.1]{Hartshorne}.

\end{remark}

A correspondence with a core fits in a Galois diagram.
If the index is finite, the intersection has transcendence degree one over $k$ and it is the function field $k(P)$ of a curve $P$. The containment diagram between function fields translates into the following commutative diagram of curves
\begin{equation}
{\label{diagram: small field theory descending lemma under Galois hp}}
\begin{tikzcd}[column sep={1.3cm,between origins},row sep={1.3cm,between origins}]
    & Z \arrow[dr, "\alpha_W"] \arrow[dl, "\alpha_Y"'] & \\
            Y \arrow[dr, "f_Y"'] & & W \arrow[dl, "f_W"] \\
    & P &
\end{tikzcd}
\end{equation}
The curve $P$ is the push-out of the correspondence diagram in the category of curves and, as such, it might not exist (see \Cref{example: correspondence degree 3}). We can replace $P$ with a genus zero curve by composing at the bottom of diagram \Cref{diagram: small field theory descending lemma under Galois hp} with any non-constant map $P \to \PP^1$.
We can replace $Z$ with the Galois closure $\bar{Z} \to P$ of $Z \to P$ to get the desired Galois diagram.

\begin{remark}
\label{remark: extending a correspondence}
When we compose at the top of the diagram with a map $f\colon Z' \to Z$, we get a new correspondence $Y \leftarrow Z' \to W$ that induces the same map $\Jac W \to \Jac Y$ up to isogeny: indeed, we have
    \[
    (\alpha_Y f)_\ast \circ (\alpha_W f)^\ast
    = (\alpha_Y)_\ast \circ f_\ast \circ f^\ast \circ (\alpha_W)^\ast
    = (\alpha_Y)_\ast \circ [\deg f]_{\Jac(Z)} \circ (\alpha_W)^\ast.
    \]
\end{remark}

\begin{remark}
\label{remark: genus at least 2 and core implies Galois}
    If the correspondence $\alpha$ has a core, it fits in a Galois diagram.
    Conversely, if $g(Z) \geq 2$ and the correspondence $\alpha$ is dominated by a correspondence $\alpha'$ with both maps Galois, then $\alpha$ has a core. There are two subgroups $H_Y, \, H_W \subseteq \Aut(Z')$ such that $Z'/H_Y = X$ and $W = Z'/H_W$. The smallest subgroup containing both is 
    $G = \langle H_Y,\, H_W \rangle \subseteq \Aut(Z')$.
    Since $\Aut(Z')$ is finite \cite[Section 22.7.9]{foag}, so is $G$. It follows that $Z' \to Z'/G=P$ is finite.
\end{remark}

Suppose that the correspondence
\(
    \alpha \colon X \leftarrow Z \rightarrow W
\)
admits a core, and let $P$ be as in diagram
\Cref{diagram: small field theory descending lemma under Galois hp}.
Let $P'$ denote the push-out of the correspondence at the level of underlying topological spaces:
\begin{equation*}
\label{eq: topological pushout of correspondence}
P' = |X| \sqcup |W| \big/ \{\, (\alpha_X(z),\, \alpha_W(z)) \mid z \in Z \,\}.
\end{equation*}
By the universal property of the push-out, there is a factorization
\(
|Z| \rightarrow P' \rightarrow |P|.
\)
Fix a point $z \in Z$ and define inductively the subsets
\[
\mathcal{F}_0(z) := \{z\}, \qquad
\mathcal{F}_{i+1}(z)
:= \alpha_X^{-1}\!\bigl(\alpha_X(\mathcal{F}_i(z))\bigr)
   \cup
   \alpha_W^{-1}\!\bigl(\alpha_W(\mathcal{F}_i(z))\bigr).
\]
All points in
\(
\mathcal{F}(z) := \bigcup_{i \ge 0} \mathcal{F}_i(z)
\)
are identified by the map $|Z| \to P'$, and hence also by the algebraic push-out morphism
$f \colon Z \to P$.
It follows that
\[
\#\mathcal{F}(z) \le \deg f
\qquad \text{for all } z \in Z.
\]
For a subset $S \subseteq Z$ of points, we define
$\mathcal{F}_n(S)$ and $\mathcal{F}(S)$ as in
\Cref{eq: definition of chaotic set}, or equivalently by
\(
\mathcal{F}_n(S) := \bigcup_{s \in S} \mathcal{F}_n(s).
\)
A finite subset of the form $\mathcal{F}(S)$ is called a \emph{clump}~\cite[Definition~9.1]{Raju}.

\begin{proposition}
    \label{prop: theorem D in the complex setting}
    Let $\alpha\colon X \leftarrow Z \to W$ be a correspondence between smooth compact Riemann surfaces. If $\alpha$ has a core, then there is a positive integer $m$ such that all points $z \in Z$ satisfy the inequality 
    \begin{equation}
    {\label{eq: criterion inequality}}
        \#\mathcal{F}(z) \leq m,
    \end{equation}
    If there is a positive integer $m$ such that inequality \Cref{eq: criterion inequality} holds for infinitely many $z \in Z$, then $\alpha$ has a core.
\end{proposition}

We divide the proof of \Cref{prop: theorem D in the complex setting} in two lemmas.
The first one states that the set of points satisfying an inequality of the form \eqref{eq: criterion inequality} for a fixed $m$, is either finite or the whole curve. 

\begin{lemma}
{\label{lemma: Fc is Zariski closed}}
    Fix an integer $M>0$. The set of points
    \[
        C = \{ z \in Z \mid \#\mathcal{F}(z) \leq M \} \subseteq Z
    \]
    is Zariski closed.
\end{lemma}
\begin{proof}
    We reduce to proving that, for any integer $n \geq 0$, the set
    \[
        C_n = \{z \in Z \mid \#\mathcal{F}_n(z)\leq M \}
    \]
    is Zariski closed. Indeed, this implies that $C = \bigcap_n C_n$ is also closed.
    We construct a proper algebraic curve $Z_n$ and a flat (proper) morphism $h \colon Z_n \to Z$ such that, for any $z \in Z$, one has $\# h^{-1}(z) = \# \mathcal{F}_n(z)$.
    The statement then follows from the lower semi-continuity of $z \mapsto \# h^{-1}(z)$ \cite[\href{https://stacks.math.columbia.edu/tag/0BUI}{Lemma 0BUI}]{stacks-project}.
    
    A \emph{path $\gamma$ of length $n$} is an element of the set $\{\alpha_X, \, \alpha_W\}^n$, which we can represent as a zig-zag diagram. For example, $\gamma=(\alpha_X, \alpha_X, \alpha_W)$ is a path of length $n=3$, which we can depict as the diagram
    \begin{equation*}
    {\label{diagram: zig-zag diagram}}
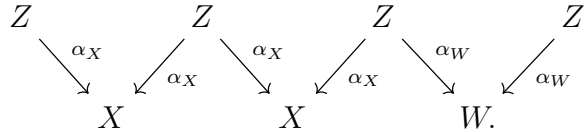
    \begin{tikzcd}[column sep = small]
        Z \arrow[dr, "\alpha_X"] && \arrow[dr, "\alpha_X"] \arrow[dl, "\alpha_X"] Z && \arrow[dl, "\alpha_X"] Z \arrow[dr, "\alpha_W"] && \arrow[dl, "\alpha_W"] Z \\
        & X && X && W. &
    \end{tikzcd}
    \end{equation*}
    Let $Z_\gamma$ be the limit of the zig-zag diagram associated with $\gamma$ in the category of schemes, which we can also define as the subschemes
    \[
        \{(z_0, \dots, z_n) \in Z^{n+1} \mid \gamma_i(z_i)=\gamma_i(z_{i-1}) \quad \forall i =1, \dots, n \} \subseteq Z^{n+1}.
    \]
    Since we get $Z_\gamma$ as the iterated fiber product of branched covers of curves, it has pure dimension 1 and every irreducible component dominates $Z$ (through the projection onto the first factor).  
    Let $Z_\gamma^1 \subseteq Z \times Z$ be the projection of $Z_\gamma$ onto the first and last factors. Notice that every connected component of $Z_\gamma^1$ dominates the first component $Z$. In particular $\dim Z_\gamma^1 = 1$.
    Let
    \[ \textstyle Z_n = \bigcup_\gamma Z_\gamma^1 \subseteq Z \times Z \]
    be the union over all paths $\gamma$ of length $n$: we take the reduced structure on the set-theoretic union. Since we are taking a finite union of curves, $Z_n$ is a (singular) curve.
    The projection $h\colon Z_n \subseteq Z \times Z \to Z$ onto the first component is a flat morphism by \cite[Chapter 4, Proposition 3.9]{Liu}: $Z$ is connected and smooth (thus regular, thus Dedekind \cite[Chapter 4, Example 2.9]{Liu}) and every irreducible component of $Z_n$ dominates $Z$ (by construction).
    Furthermore the fiber $h^{-1}(z)$ of a point $z \in Z$ is the set of pairs $(z_0, z_{n})$ such that $z_0=z$ and there is a path $\gamma$ of length $n$ connecting $z_0 \mapsto z_{n}$. By definition, the set of such $z_n$ is the set $\mathcal{F}_n(z)$. Hence $\# h^{-1}(z) = \# \mathcal{F}_n(z)$.
\end{proof}

We establish the other implication of \Cref{theorem: criterion} via complex-analytic methods. We note that \cite[Theorem 9.6]{Raju} yields a stronger result in this direction: it suffices to find two distinct non-ramified points with $\mathcal{F}(z_1) \neq \mathcal{F}(z_2)$ to conclude that $\alpha$ has a core. Since our approach proceeds by a different strategy, we include the argument here for completeness.

\begin{lemma}
{\label{lemma: sequence of open sets}}
    Let $Z, \, X$, and $W$ be smooth connected (not necessarily compact) Riemann surfaces.
    Let $f\colon Z \to X$ and $g\colon Z \to W$ be holomorphic maps with no ramification points. Fix an integer $n \geq 0$. Let $z_0, \, z_1, \, \dots, \, z_n \in Z$ be a sequence of points such that
    \begin{equation}
        {\label{eq: weird set-theoretic condition for the sequence}}
        f(z_i)=f(z_{i-1}) \quad\text{or}\quad g(z_i)=g(z_{i-1}) \qquad\forall\, i = 1, \, \dots,\, n. 
    \end{equation}
    Then, there exists a sequence of open connected neighborhoods $z_i \in U_i \subseteq Z$ such that:
    \begin{itemize}
        \item if $f(z_i)=f(z_{i-1})$, then $f^{-1}f\colon U_{i-1} \to U_i$ is an isomorphism (of Riemann surfaces);
        \item if $f(z_i)\neq f(z_{i-1})$, then $g(z_i)=g(z_{i-1})$ and $g^{-1}g\colon U_{i-1} \to U_i$ is an isomorphism.
    \end{itemize}
    Furthermore, given two sequences $(z_i)_i$ and $(z_i')_i$ such that $z_0=z_0'$ and $z_n \neq z_n'$, we can choose sequences $(U_i)_i$ and $(U_i')_i$ with the property above and such that $U_n\cap U_n' = \emptyset$.
\end{lemma}
\begin{proof}
    We prove the statement by induction on $n$.
    For $n=0$ there is nothing to prove.
    Fix a sequence $z_0, \, \dots, \, z_n$.
    Applying the inductive hypothesis to $z_0, \, \dots, \, z_{n-1}$ we get isomorphic connected neighborhoods $V_0, \, \dots, \, V_{n-1}$ with the property in the statement.
    Up to swapping the roles of $f$ and $g$, we can assume that $f(z_n) = f(z_{n-1})$. Since $f$ is holomorphic, it is an open map \cite[Corollary 3.2.3]{Tamas}; therefore $f(V_{n-1}) \subseteq X$ is open. Let $T \subseteq f(V_{n-1})$ be a trivializing connected open neighborhood for the cover $f$. Let $U_n$ be the sheet over $T$ containing $z_n$ and $U_{n-1}$ the sheet containing $z_{n-1}$. By construction, $U_{n-1} \subseteq V_{n-1}$. Thus the chain of isomorphisms connecting $V_{n-1}$ to $V_0$ defines isomorphic connected subsets $U_i \subseteq V_i$ for all $i \leq n-1$. Furthermore, $U_{n-1}$ and $U_n$ are isomorphic.
    
    For the last part of the statement, notice that it is enough to construct the two sequences simultaneously and, if necessary, restrict the trivializing neighborhoods further to make them disjoint.
\end{proof}

\begin{proof}[Proof of \Cref{prop: theorem D in the complex setting}]
    Since we have already discussed the first part of the statement in the paragraph preceding the statement of the proposition, we will focus on the second part.
    \Cref{lemma: Fc is Zariski closed} states that the set of points $z \in Z$ for which inequality \Cref{eq: criterion inequality} holds is a Zariski closed subset $C \subseteq Z$. As $Z$ is an irreducible connected curve, and any infinite closed subscheme $V \subseteq Z$ must be one-dimensional, we conclude that $C=Z$.
    
    Let $d$ be the smallest integer such that
    $C_d = \{ z \in Z \mid \#\mathcal{F}(z) \leq d \}$ is infinite (hence equals to $Z$). Then all but finitely many points of $Z$ satisfy $\#\mathcal{F}(z) = d$ (i.e.~all points in $Z^\circ =  C_d- C_{d-1}$).
    Let $\beta \colon |Z| \to P$ be the topological push-out map, as in \Cref{eq: topological pushout of correspondence}.
    We claim that the restriction of $\beta$ to $|Z^\circ| \subseteq |Z|$ gives a complex structure to $P^\circ = \beta(|Z^\circ|) \subseteq P$ for which $\beta\colon Z^\circ \to P^\circ$ is holomorphic. It follows that its completion $\beta\colon Z \to P$ is a branched cover of compact Riemann surfaces, as in \cite[Proposition 3.2.9]{Tamas}.
    
    This proves that $\alpha$ has a core, hence it suffices to prove the claim. Fix a point $z \in Z^\circ$, for which $\#\mathcal{F}(z) = d$, and a positive integer $n$ such that $\#\mathcal{F}_n(z) = \#\mathcal{F}(z)$.
    For each $y \in \mathcal{F}_n(z)$ fix a sequence $(z_i)_i$ such that $z_0=z$ and $z_n = y$. \Cref{lemma: sequence of open sets} gives connected open subsets $z \in U_0(y)$, $y \in U_n(y)$ that are isomorphic through a chain of covers. Furthermore, we can assume all $U_n(y)$ to be disjoint.
    Any point $z'$ in the intersection $V = \bigcap_y U_0(y)$ has the following property by construction: $\#\left(\mathcal{F}_n(z') \cap U_n(y)\right) = 1$ for any $y \in \mathcal{F}_n(z)$. This implies that
    \[
       \textstyle \mathcal{F}(V) = \mathcal{F}_n(V) = \bigcup_y V_n(y),
    \]
    where $V_n(y)$ is the image of $V$ via the fixed isomorphism $U_0(y) \to U_n(y)$. 
    The push-out of the restriction of $\alpha$ to $\bigcup_y V_n(y)$ corresponds to a single isomorphic copy of $V$.
    This gives a trivializing neighborhood $\beta (V)$ for $\beta(z)\in P$, which inherits a complex structure from any of its sheets (since they are all isomorphic).
    By repeating the argument for every $z \in Z^\circ$, we cover $P^\circ = \beta(|Z^\circ|)$ with trivializing neighborhoods for $\beta$. This cover defines a complex atlas on $P^\circ$ for which $\beta$ is holomorphic.
\end{proof}

\Cref{theorem: criterion} follows from \Cref{prop: theorem D in the complex setting} and \Cref{remark: having a core is a geometric property}.

\section{The Prym correspondence has a core}
\label{section: conclusions}
We use the criterion of \Cref{theorem: criterion} to deduce \Cref{theorem: main theorem} from \Cref{theorem: the correspondence gives a Prym isogeny}.
Denote by $\alpha\colon X \leftarrow Z \to W$ the correspondence in \Cref{eq: correspondence in the UZ setting} with $\alpha_X= \pi\pr_1$ and $\alpha_W=\beta$. 

\begin{corollary}[\Cref{theorem: main theorem}]
    The correspondence $\alpha$ in \Cref{theorem: the correspondence gives a Prym isogeny} fits in a Galois diagram. The core comes from the degree 2 map $x \colon X \to \PP^1$ introduced in \Cref{definition: associate P1 cover phi}.
\end{corollary}
\begin{proof}
    Fix a point $z \in Z(\bar k)$. Consider the set
    \[ F(z) = \alpha_X^{-1}\alpha_X(z) \cup  \alpha_X^{-1}\iota\alpha_X(z) = \{ z' \in Z \mid \alpha_X(z')=\alpha_X(z) \text{ or } \alpha_X(z')=\iota\alpha_X(z) \}. \]
    We claim that we have $\mathcal{F}_n(z) \subseteq F(z)$ for all $n \geq 0$.
    Since $\mathcal{F}_0(z) \subseteq F(z)$, it is enough to prove that $F(z)$ is invariant under the set-theoretic operations $\alpha_X^{-1}\alpha_X$ and $\alpha_W^{-1}\alpha_W$. The former invariance is trivial.
    For the latter invariance observe that $\alpha_W^{-1}\alpha_W(z')$ is the set $\{ z', \, \eta(z') \}$, where $\eta$ is the involution induced by $Y^2 \to Y^2, \, (y_1, y_2) \mapsto (y_2, y_1)$. We need to show that $\eta(z') \in F(z)$ for any $z' \in F(z)$. This follows from the defining equation \Cref{eq: equatons for W}: indeed this identity implies 
    \[
        \alpha_X(z')= \pi \pr_1(z')=\pi(y_1)=\iota\pi(y_2) = \iota\pi\pr_1\eta(z') = \iota\alpha_X\eta(z').
    \]
    It follows from \Cref{theorem: criterion} that $\alpha$ has a core, since $$\#\mathcal{F}(z) \leq \#\alpha_X^{-1}\alpha_X(z) + \#  \alpha_X^{-1}\iota\alpha_X(z) = 2 \cdot \deg(\alpha_X) = \deg(\alpha_X)\cdot\deg(\alpha_W)$$ for all $z \in Z(\bar k)$.
    Therefore, the correspondence fits in a Galois diagram.
    
    For the second part of the statement, notice that the map from $X$ to the core must identify $\xi \in X(\bar k)$ with $\iota(\xi)$ and no other point. It thus coincides with the map $x\colon X \to \PP^1$ by \Cref{definition: associate P1 cover phi}.
\end{proof}

Finally, we show that not every correspondence between curves has a core. This makes \Cref{conjecture} interesting from a theoretical point of view, in addition to the already established practical utility.

\begin{example}
{\label{example: correspondence degree 3}}
    Let $\pi_X\colon Y \to X$ be the cover of \Cref{example: degree 3}.
    Let $W$ be the complementary curve to $\pi_X$ of \cite[Section 1]{Kuhn}, which comes with a ramified degree-3 cover $\pi_W \colon Y \to W$. We give the maps of the correspondence $\pi\colon X \leftarrow Y \to W$ in terms of the associated $\PP^1$-covers:
    \begin{equation*}
        \phi_X(t) = \frac{t^2}{t^3-4}, \qquad \phi_W(t) = \frac{-16 t^3 + 24 t^2 - 9 t}{256 \cdot(t^3 - 1)}.
    \end{equation*}
    Let $y=(2, s_y) \in Y(\C)$. If $\pi$ has a core, then \Cref{theorem: criterion} implies that $\#\mathcal{F}(y)$ is bounded.
    We get a lower bound on the cardinality of the sets $\mathcal{F}_i(y)$ by looking at the $t$-coordinates of its points: that is, by making the computation for the algebraic correspondence $\phi\colon \PP^1 \leftarrow \PP^1 \to \PP^1$ on the associated $\PP^1$-covers.
    A quick computation with a computer algebra system \cite{OurScripts} yields $$\#\mathcal{F}_{2}(y) \geq 18 \quad \text{and}\quad \#\mathcal{F}_{9}(y) \geq 68 \, 883,$$
    suggesting that the correspondence does not have a core.
\end{example}

\begin{remark}
\label{remark: the monodromy construction does not give the complementary cover directly}
    More generally, let $\pi_X\colon Y \to X$ be an irreducible degree $n$ cover between curves of respective genera $g(Y)=2$ and $g(X)=1$, and denote by $\pi_W\colon Y \to W$ the complementary degree $n$ cover of \cite[Section 1]{Kuhn}.
    We expect that the strategy of \Cref{example: correspondence degree 3} can always be applied to show that the correspondence $\pi\colon X \leftarrow Y \to W$ does not have a core.
\end{remark}

\section{Related questions}
\label{section: related questions}
Several related questions naturally arise from the problems discussed in the previous sections. While some of these perspectives provide theoretical solutions to more general problems, they are often less effective computationally. In contrast, the present work focuses on approaches tailored to split Jacobians, which are computationally more feasible and express the problems in terms of the more approachable geometry of the underlying curves. These considerations provide a broader context in which the main results fit; for completeness, we briefly record them here.


\subsection{Pseudo Brauer relations}
Let $G \subseteq \Aut(\bar{Y})$ be a finite group of automorphisms of a curve $\bar{Y}$. A \emph{Brauer relation} is a formal sum $\sum_i H_i - \sum_j H_j'$ of subgroups $H_i, \, H_j' \subseteq G$ such that the representations $\bigoplus_i \Ind_{H_i}^G(1_{H_i})$ and $\bigoplus_j \Ind_{H_j'}^G(1_{H_j'})$ are isomorphic.

Similarly, a \emph{pseudo Brauer relation} is a formal sum of subgroups that comes from a relaxed isomorphism condition \cite[Definition 3.1 and Definition 4.1]{dokchitser2024parityranksjacobianscurves}. There is a functorial construction \cite[Theorem 4.14]{konstantinou2024galoiscoverscurvesarithmetic} that, given a pseudo Brauer relation $\sum_i H_i - \sum_j H_j'$, yields an isogeny
\begin{equation}
    {\label{eq: isogeny induced by pbr}}
    \textstyle \prod_i \Jac(\bar{Y}/H_i) \sim \prod_j \Jac(\bar{Y}/H_j').
\end{equation}
Conversely, an isogeny between abelian varieties is said to be \emph{pseudo Brauer verifiable} if it can be written in the form \Cref{eq: isogeny induced by pbr} for suitable choices of $\bar{Y}, \, G$, and a pseudo Brauer relation.
This result streamlines the important step, already present in the Galois theory approach \cite[Section 3]{lombardo2020decomposing}, of translating the $G$-representation structure of the tangent space of $\Jac\bar{Y}$ into isogenies.

In this language, \Cref{conjecture} states that the isogeny in \Cref{eq: first instance of corr diagram} is pseudo Brauer verifiable (through a prescribed curve $\bar{Y}$). In these terms, the conjecture has been proven, in the setting of \Cref{section: the problem}, for covers $\pi \colon Y \to X$ whose degree $n$ is a prime number and whose ramification type is generic \cite[Theorem~7.8]{konstantinou2023}. \Cref{theorem: main theorem} extends the results to all covers between curves of genera 2 and 1, independently of the degree and the ramification type.

\subsection{Isogenous factors}

Conjecturally, over a number field $k$, there exists a uniform bound $N$ such that two elliptic curves over $k$ are isogenous if and only if their modulo-$N$ Galois representations are isomorphic \cite{bakker2015freymazurconjecturelowgenus}.
In the setting of \Cref{section: the problem}, the two factors of $\Jac Y$ have isomorphic modulo-$n$
Galois representations, where $n = \deg \pi$.
A natural question is therefore whether the two elliptic curves $\Jac X$ and $\Jac W$ are isogenous.
One of the original motivations for studying split Jacobians was precisely to address this question
\cite{Frey}.

Given two non-isogenous elliptic curves $E$ and $E'$ and an integer $n \geq 2$, it is always possible to
construct a degree-$n$ cover $Y \to E$ between curves of respective genera $2$ and $1$ that induces an
isogeny $\Jac Y \sim E \times E'$ \cite{MR1085258}.
Moreover, the generic point of the moduli space of two-dimensional split Jacobians corresponds to a product
of two non-isogenous elliptic curves \cite{Kani1997}.
Thus, in general, the two factors are not isogenous.

Given a cover $Y \to X$ as in \Cref{section: the problem}, deciding whether the two elliptic factors of
$\Jac Y$ are isogenous can be done using standard methods.
For instance, one may first obtain explicit bounds on the degree $N$ of a possible
isogeny \cite{MR1217345}, and then test whether the pair of $j$-invariants is a zero of a modular polynomial
of level at most $N$.
This strategy is employed in \cite{MR4932442} to compute models for special families in the moduli space of
$n$-split Jacobians with $N$-isogenous factors, for small values of $n$ and $N$.

\subsection{Explicit Poincaré decomposition}
Let $A$ be an abelian variety defined over the number field $k \subseteq \C$.  
The Poincaré complete reducibility theorem asserts that $A$ is isogenous (over $k$) to a product of simple abelian varieties, \( A \sim A_1 \times \cdots \times A_r \).
Several approaches exist for making this decomposition effective.
Although our work focuses on methods tailored to {split Jacobians}, for which more specialized techniques are both more efficient computationally and more transparent theoretically, we report them for completeness.

Once a basis of holomorphic differential $1$-forms $\{\omega_i\}_i$ and a symplectic basis of integral homology $\{\gamma_j\}_j$ are fixed, the \emph{period matrix} of $A$ is the complex matrix with entries
\[
\Pi_{ij} \;=\; \int_{\gamma_j} \omega_i.
\]
One can exploit this transcendental data to determine the Poincaré decomposition.
Indeed, one can compute the geometric endomorphism ring $\End(A_{\bar{k}})$ following~\cite{MR3882288, MR3904148, MR4280568}, and then deduce the decomposition of $A$ from the structure of its endomorphism algebra~\cite{WedderburnDecomposition}.
In the setting of \Cref{section: the problem}, it is easy to compute the period matrix of $\Jac Y$ explicitly and apply the above strategy. 

A more direct method, due to~\cite{RODRIGUEZ2025364}, attempts to find a change of basis of the $\{\omega_i\}_i$ and $\{\gamma_j\}_j$ that places the period matrix $\Pi$ in a block-diagonal form. This search can be reduced to solving a system of equations over the rational numbers. In turn, any such change of basis provides partial information on the endomorphism algebra, specifically yielding the desired idempotents.
This approach reduces the decomposition problem to solving a Diophantine equation: while finding rational solutions is easy, proving the non-existence of solutions can be extremely difficult.
This is not an obstacle when $A$ is known to be split, as in the setting of \Cref{section: the problem} with $A=\Jac Y$, since one is guaranteed to keep finding rational solutions until $A$ is fully decomposed.

\bibliographystyle{alpha}
\bibliography{sample}
\end{document}